\documentclass{amsart}

\usepackage{amssymb, graphicx}

\usepackage{amsfonts}
\usepackage{latexsym}
\usepackage{amscd}
\usepackage{verbatim}

\allowdisplaybreaks[1]

\newtheorem{theorem}{Theorem}[section]
\newtheorem{proposition}[theorem]{Proposition}
\newtheorem{lemma}[theorem]{Lemma}
\newtheorem{corollary}[theorem]{Corollary}
\theoremstyle{remark}
\newtheorem{remark}[theorem]{Remark}
            
\newtheorem{example}[theorem]{Example}

\newcommand{\s}{\sigma} 
\newcommand{\CC}{\mathbb{C}}
\newcommand{\R}{\mathbb{R}}
\newcommand{\A}{\hat{a}}
\newcommand{\B}{\hat{b}}
\newcommand{\e}{\epsilon}
\newcommand{\U}{\mathcal{U}}
\newcommand{\V}{\mathcal{V}}
\newcommand{\BB}{\overline{B}}

\newcommand{\ls}{\langle}
\newcommand{\rs}{\rangle}
\newcommand{\tnz}{\otimes}

\newcommand{\bz}{\setminus}
\newcommand{\uu}{\xi}
\newcommand{\co}{\mathsf{c}}
 \newcommand{\spn}{\mathrm{span}}

\newcommand{\soc}{\mathrm{soc}}

\begin{document}
\title[Identifying derivations through  the spectra of their values]{Identifying
derivations through  the spectra of their values}

\author{Matej Bre\v sar}
\author{Bojan Magajna}
\author{\v Spela \v Spenko}
\address{M. Bre\v sar,  Faculty of Mathematics and Physics,  University of Ljubljana,
 and Faculty of Natural Sciences and Mathematics, University
of Maribor, Slovenia} \email{matej.bresar@fmf.uni-lj.si}
\address{ B. Magajna, Faculty of Mathematics and Physics,  University of Ljubljana, Slovenia} 
\email{bojan.magajna@fmf.uni-lj.si}
\address{\v S. \v Spenko,  Institute of  Mathematics, Physics, and Mechanics,  Ljubljana, Slovenia} \email{spela.spenko@imfm.si}

\begin{abstract}
 We consider the relationship between derivations $d$ and $g$ of a Banach algebra
$B$ that satisfy $\s(g(x)) \subseteq \s(d(x))$ for every $x\in B$, where $\s(\, .
\,)$ stands for the spectrum. It turns out that in some basic situations, say if
$B=B(X)$, the only possibilities are that $g=d$, $g=0$, and, if $d$ is an inner
derivation implemented by an algebraic element of degree $2$, also $g=-d$. The
conclusions in more complex classes of algebras are not so simple, but are of a
similar spirit. A rather definitive result is obtained for von Neumann algebras. In
general  $C^*$-algebras we have to make some adjustments, in particular we restrict
 our attention to
 inner derivations implemented by selfadjoint elements. We also consider a related
condition  $\|[b,x]\|\leq M\|[a,x]\|$  for all selfadjoint elements $x$ from a
$C^*$-algebra  $B$, where $a,b\in B$  and $a$ is   normal.
\end{abstract}

\thanks{2010 {\em Math. Subj. Class.} 46H05,  46L57, 47B47. }
\thanks{Supported by ARRS Grant P1-0288.}
\maketitle

\section{Introduction}

Let $B$ be a Banach algebra. By $\s(x)$ (resp. $\rho(x)$) we denote the spectrum (resp. spectral radius) of $x\in B$.
Finding conditions under which a linear map $f$ from $B$ onto another Banach algebra that satisfies  $\s(f(x))=\s(x)$  for every $x\in B$ is necessarily
 a Jordan homomorphism is an important and widely open problem, raised by  Kaplansky \cite{K}.
  Motivated by 
some of its aspects, the first and the third author have recently considered the question whether the equality $\s(ax) = \s(bx)$ for every $x\in B$, where $a,b\in B$ are fixed elements, implies $a=b$ \cite{BS}. An affirmative answer has been obtained for some classes of algebras, including  $C^*$-algebras.  
Now one may wonder  what can be said about 
other pairs of linear maps on $B$ such that
 the spectra of their values coincide  at each point. In this paper we will consider a pair of derivations $d$ and $g$. In fact, in most of our results we will not need to assume the equality of the spectra, but only
\begin{equation}\label{ena}\tag{S}
\s(g(x)) \subseteq \s(d(x)) \quad\text{for all  $x\in B$.} 
\end{equation}
By studying \eqref{ena} we  follow the line of investigation of spectral properties of values of derivations. 
Let us mention some topics in this area:  
   derivations and their products that have quasinilpotent values \cite{ CKL, Lee,   TS}, spectrally bounded derivations \cite{ BrMa}, and derivations all of whose values have a finite spectrum  \cite{ BoMa,BoSe,BrSe0}.
  A topic of a  different kind, which, however, is closer to the problem considered here than it may 
  seem at a first glance, is the study of  derivations $d$ and $g$ such that the range of $g(x)$ is contained in the range of $d(x)$ for every $x\in B$. In their seminal work \cite{JW}, Johnson and Williams considered such a range inclusion for the case
   where $B=B(H)$ and $d$ is an 
   inner derivation implemented by  a normal operator.  See also \cite{Fong, KS1} for further development.
   
   When can \eqref{ena} occur? A trivial possibility is that $g=d$.  If the range of $d$ consists of 
   non-invertible elements, e.g., if $d$ is an inner derivation implemented by an element from a 
   proper ideal, then we can take $g=0$. There is another, less obvious possibility: $g=-d$, where $d$ 
   is inner and implemented by an algebraic element of degree $2$. This situation is briefly discussed in Section 2. In Section 3 we show that the aforementioned three possibilities are also the only ones if $B$ is a primitive Banach algebra with nonzero socle. Using this result we are able to handle \eqref{ena} for a general semisimple Banach algebra $B$ under the assumption that $d(x)$ has a finite spectrum for every $x\in B$. Section 4 is devoted to the case where $B$ is a von Neumann algebra.  The main result  says that  \eqref{ena} implies that
$B$ can be decomposed into three parts such that on each of them one of the  possibilities $g=d$, $g=0$, and $g=-d$ holds.   In Section 5 we consider the case where $B$ is a $C^*$-algebra and derivations are inner, $d:x\mapsto [a,x]$ and $g:x\mapsto [b,x]$. First we show that $b$ lies in $\{a\}''$, the (relative) bicommutant of $\{a\}$, provided that $a$ is normal. Other results are of a slightly different nature. Consider the condition
 $\rho([b,x])\le M\rho([a,x])$ which clearly follows from  \eqref{ena} (with $M=1$). If both $a,b$ are selfadjoint, then the commutators $[a,x]$ and $[b,x]$ are anti-selfadjoint whenever $x$ is selfadjoint. In this case we can therefore write   our condition     as 
 \begin{equation}\label{dva}\tag{N}
||[b,x]||\le  M||[a,x]||  \quad\text{for all selfadjoint $x\in B$.} 
\end{equation}
 According to \cite[Lemma 1.1]{JW}, we can view \eqref{dva} as a dual problem to the range inclusion 
 problem. Following \cite{JW} and consecutive papers \cite {Fong, KS1} we consider  the condition 
 \eqref{dva} for a normal element $a$ in a $C^*$-algebra $B$.     The most complete result, however, is obtained for  selfadjoint elements $a,b$ under the assumption  that the equality (with $M=1$) holds in \eqref{dva} (equivalently, 
 $\rho([b,x]) = \rho([a,x])$ for  all selfadjoint $x\in B$).

A word about terminology and notation. By a Banach algebra we shall mean a complex Banach algebra. For simplicity we assume that all our algebras have  identity elements. 
 We write $Z(B)$ for the center of $B$.

\section{Commutators with symmetric spectra}  

If an element $a$ from a Banach algebra $B$ is such that for every $x\in B$ the spectrum of $[a,x]$ is symmetric in the sense that $\s\bigl([a,x]\bigr)=-\s([a,x])$, then the spectral inclusion condition \eqref{ena} is fulfilled for $g=-d$, $d=[a,-]$. This situation is of special interest. Let us show that it can indeed occur.

\begin{lemma}\label{alg}
Let $B$ be a Banach algebra. 
If $e\in B$ is an idempotent and $x\in B$ is arbitrary, then $[e,x]$ is similar to $-[e,x]$. In particular, $\s([e,x])=-\s([e,x])$.
\end{lemma}

\begin{proof}
Take $s=1-2e$. Then $s=s^{-1}$ and $s[e,x]s^{-1}=-[e,x]$.
\end{proof}

\begin{lemma}\label{alg2}
Let $B$ be a Banach algebra in which for all $x,y\in B$, $xy=1$ implies $yx=1$.  
If $a\in B$ is algebraic of degree 2, then $\s([a,x])=-\s([a,x])$ for all $x\in B$. 
\end{lemma}

\begin{proof}
Replacing $x$ by $\alpha x+\beta$, where $\alpha$ and $\beta$ are suitable scalars, the proof reduces to two cases: $a$ is an idempotent and $a$ is nilpotent of nilpotency degree 2.  As the first one has already been treated in Lemma \ref{alg}, we may suppose that $a^2=0$. It suffices to prove that $-1\in \s([a,x])$ implies $1\in \s([a,x])$ for every $x\in B$.
We have 
$$
(1+xa)(1-ax)=1-[a,x], 
$$
and
$$
-(1-xa)(1+ax)=-1-[a,x].
$$
If $1-[a,x]$ is invertible, then, by our assumption, $1+xa$ and $1-ax$ are invertible. Consequently, 
(as it is well-known) $1+ax$ and $1-xa$ are invertible. By the above equality this implies
the invertibility of $-1-[a,x]$. 
\end{proof}

Let us show that the assumption that  $xy=1$ implies $yx=1$ is not redundant.
We give an example of an element $a \in B(H)$  such that $a^2=0$ and there exists $x\in B(H)$ such that $\s([a,x])$ is not symmetric.

\begin{example}
Let $a,x\in B(H)\cong B(H\bigoplus H)\cong M_2(B(H))$ be of the form
$$
a=
\left( \begin{array}{cc}
0 & 1 \\
0 & 0
\end{array}
\right),
\quad
x=
\left( \begin{array}{cc}
y& 0 \\
u & 0 
\end{array}
\right).
$$
Then 
$$
[a,x]=
\left( \begin{array}{cc}
u & -y \\
0 & -u
\end{array}
\right),
\quad
1+[a,x]=
\left( \begin{array}{cc}
1+u & -y \\
0 & 1-u 
\end{array}
\right),
\quad
1-[a,x]=
\left( \begin{array}{cc}
1-u & y \\
0 & 1+u 
\end{array}
\right).
$$
It suffices to find $u$ such that $1-u$ is right invertible, $1+u$ is left invertible and $\{0\}\neq \ker(1-u)$ is isomorphic 
to $\mathrm{im}(1+u)^\bot$. 
For then we take for $y$ an operator which is an isomorphism from $\ker(1-u)$ onto $\mathrm{im}(1+u)^\bot$ and is zero on $\ker(1-u)^\bot$.
Then $1+[a,x]$ is bijective (hence invertible), but $1-[a,x]$ is not injective since $\ker(1-u)\neq \{0\}$. One possible choice
for $u$ is 
$\left( \begin{array}{cc}
1+s^* & 0 \\
0 & -1+s
\end{array}
\right),$
where $s$ is the unilateral shift and $s^*$ its adjoint. Indeed, $1-u=
\left( \begin{array}{cc}
-s^* & 0 \\
0 & 2-s
\end{array}
\right)$ is right invertible (since $s^*$ is right invertible and $2-s$ is invertible), $1+u=
\left( \begin{array}{cc}
2+s^* & 0 \\
0 & s
\end{array}
\right)$ is left invertible and $\ker(1-u)\cong \ker s^*=\mathrm{im}s^\bot\cong\mathrm{im}(1+u)^\bot.$
\end{example}

Roughly speaking, we have shown that algebraic elements of degree 2 often generate derivations whose values have symmetric spectra. Later, in Lemma \ref{alg3}, we shall see that these are also the only natural examples of such elements.

\section{Results on Banach algebras}

We begin our discussion on the spectral inclusion condition \eqref{ena} in algebras in which the spectrum can be easily computed at least for some elements. The prototype example we have in mind is $B(X)$, the algebra of all bounded operators on a Banach space $X$. In this case the operators of finite rank have an easily approachable spectrum. Actually, we will work in a slightly more general setting of primitive Banach algebras with nonzero socle  and we will replace \eqref{ena} with a technically weaker condition $\s(g(x))\subseteq \s(d(x))\cup\{0\}$. This will be needed for further applications.

Recall that a semiprime  Banach algebra $B$ is said to have a nonzero socle if $B$ has minimal one-sided ideals. In this case the socle of $B$, soc$(B)$, is defined as the sum of all minimal left ideals of $B$ (equivalently, the sum of all minimal right ideals of $B$). The socle has a particularly important role if $B$ is a primitive algebra.

\begin{theorem}\label{rank1}
Let $B$ be a  primitive Banach algebra with nonzero socle. If derivations $d,g:B\to B$ satisfy $\s(g(x))\subseteq \s(d(x))\cup\{0\}$ for all $x\in {\rm soc}(B)$, then $g=\lambda d$  with $\lambda\in \{-1,0,1\}$.
\end{theorem}

\begin{proof}
Let us recall some standard facts about the structure of $B$ (see, e.g., \cite[Section 31]{Dun}). 
There exists an idempotent $e$ such that $eBe=\CC e$. We may identify $eBe$ with $\CC$.
Denote the regular  representation of $B$ on $Be$ by $\pi$, which is faithful since $B$ is primitive.  
If we define $\ls x,y\rs=yx$ for $x\in Be,y\in eB$, then  $x\mapsto \ls x,v\rs$ is a linear functional on $Be$ for every $v\in eB$. Write $X=Be$ and $Y=\{f\in X^*:f=\ls-,v\rs\; \text{for some}\; v\in eB\}$.  If $\xi_1,\dots,\xi_n\in X$ are linearly independent, there exists $f\in Y$ such that $f(\xi_1)=1$ and $ f(\xi_i)=0$ for all $i>1$. 
All operators of the form  $\sum_{i=1}^n \xi_i\tnz f_i$ with $\xi_i\in X, f_i\in Y$ are contained in $\pi({\rm soc}(B))$.   We write $\{\xi\}^\bot=\{f\in Y: f(\xi)=0\}$ for $\xi\in X$.

We may and we shall identify $B$ with $\pi(B)$.
Every derivation $\widetilde{d}$ on $B$ is of the form $[\widetilde{a},-]$ for
some $\widetilde{a}\in B(X)$. The proof is the same as for $B(X)$. (One just defines
$\widetilde{a}:\xi\mapsto \widetilde{d}(\xi\tnz f)\eta$ for some $\eta\in X$, $f\in
Y$ with $f(\eta)=1$ and details can be easily verified.) 
Hence there exist $a,b\in B(X)$ such that  
$d=[a,-]$ and $g=[b,-]$. 
We take $\uu\in X$, $f\in \{\xi\}^\bot$ and calculate the spectra of $[a,\xi\tnz f]$ and $[b,\xi\tnz f]$.
These are operators of rank at most 2, and the only nonzero elements of their spectra appear as eigenvalues of their restrictions $\bar{a},\bar{b}$ to the vector spaces $\spn\{\xi,a\xi\}$ and $\spn\{\xi,b\xi\}$, respectively.
Suppose that $f( b\xi)\neq 0$. Then $-f(b\xi)$ is a nonzero eigenvalue of $\bar{b}$ (corresponding to the eigenvector $\xi$).
Since $\bar{b}$ has trace zero, its eigenvalues are $-f(b\xi)$ and $f(b\xi)$. By the hypothesis of the theorem,  $\bar{a}$ is nonsingular with nonzero eigenvalues $-f(a\xi)$ and $f(a\xi)$, and $\{-f(a\xi),f(a\xi)\}=\{-f(b\xi),f(b\xi)\}$. Hence $f(a\xi)=\pm f( b\xi)$.

Therefore, for all $\xi\in X,f\in \{\xi\}^\bot$ one of the following possibilities holds:  
$f((a-b)\xi)=0$ or $f((a+b)\xi)=0$ or $f(b\xi)=0$ . 
We shall have established the theorem  
if we prove that the same possibility holds for all $\xi\in X,f\in \{\xi\}^\bot$. 
Indeed, if $c\in B(X)$ has the property that $f(c \xi)=0$ for all $\xi\in X,f\in \{\xi\}^\bot$, 
then  for all $\xi\in X$, $\xi$ and $c\xi$ are linearly dependent, which easily implies $c\in \CC 1$.

We first fix $\xi$ and define 
\begin{align*}
X_\xi^0 & = \{f\in \{\xi\}^\bot:  f( b\xi)=0\},\;\\
X_\xi^- & = \{f\in \{\xi\}^\bot: f((a-b)\xi)=0\},\;\\
X_\xi^+ & =\{f\in \{\xi\}^\bot: f( (a+b)\xi)=0\}.
\end{align*}
From what has already been proved it follows that $\{ \xi\}^\bot=X_\xi^0\cup X_\xi^-\cup  X_\xi^+$. 
Since  $X_\xi^0$, $X_\xi^-$ and $X_\xi^+$ are vector spaces, we may conclude  that one of them equals $\{\xi\}^\bot$.

Let 
\begin{align*}
X^0 & =\{\xi\in X:  \{\xi\}^\bot=X_\xi^0\}=\{\xi\in X:b\xi\in \CC \xi\},\;\\ 
X^- & =\{\xi\in X:  \{\xi\}^\bot=X_\xi^-\}=\{\xi\in X:(a-b)\xi\in \CC \xi\}, \;\\
X^+ & =\{\xi\in X:  \{\xi\}^\bot=X_\xi^+\}=\{\xi\in X:(a+b)\xi\in \CC \xi\}.
\end{align*}
Then $X^0, X^-, X^+$ are closed sets with union $X$, therefore at least one of them, say $X^\delta$ with $\delta\in\{0,-,+\}$, has a nonempty interior by Baire's theorem.  
Let us prove that $X^\delta$ is a vector space. Since $X^\delta$ is closed under scalar multiplication, 
it suffices to show that $\xi+\xi'\in X^\delta$ for linearly independent $\xi,\xi'\in X^\delta$. 
As $X^\delta$ has a nonempty interior, there exists $\alpha\in \CC\bz\{0\}$ such that 
$\xi+\alpha \xi'\in X^\delta$. Then $(\delta a+b)(\xi+\alpha \xi')=\lambda_{\alpha}(\xi+\alpha \xi')$
for some $\lambda_{\alpha}\in\mathbb{C}$. As $\xi,\xi'\in X^\delta$, 
we have 
$(\delta a+b)(\xi+\alpha \xi')=(\delta a+b)\xi+(\delta a+b)\alpha \xi'=
\lambda_\xi \xi+\alpha\lambda_{\xi'}\xi'$ 
for some $\lambda_\xi,\lambda_{\xi'}\in \CC$. 
The linear independence of $\xi$, $\xi'$ now gives $\lambda_\xi=\lambda_{\alpha}=\lambda_{\xi'}$, 
which implies that $\xi+\xi'\in X^\delta$. Hence $X^\delta$ is a vector subspace of $X$ with 
nonempty interior and thus $X^\delta=X$, from which we easily deduce that $\delta a+b\in\CC 1$.
\end{proof}

The following lemma shows that in the present context we have a kind of a converse to the results from the preceding section.

\begin{lemma}\label{alg3}
Let $B$ be a  primitive Banach algebra with nonzero socle. If a derivation $d:B\to B$ satisfies $\s(d(x))=-\s(d(x))$ for all $x\in B$, then $d$ is implemented by an algebraic element of degree 2.
\end{lemma}

\begin{proof}
We adopt the notation of the preceding theorem. Let $d=[a,-]$. 
Assume that $a$ is not algebraic of degree 2. 
By Kaplansky's theorem (see e.g. \cite[Theorem 4.2.7]{Aup}), we can find $\xi\in X$ such that $\xi,a\xi,a^2\xi$ are linearly independent. 
Take the element $\xi\tnz f+a\xi\tnz g$ 
where $f\in Y$ satisfies $f( \xi)\neq 0, f( a\xi)= 0, f(a^2\xi)\neq 0$ 
and $g\in Y$ satisfies $g( \xi)= 0, g( a\xi)= 0, g(a^2\xi)\neq 0$. 
Then the range of $ [a,\xi\tnz f+a\xi\tnz g]$ is contained in the vector subspace of $X$ spanned by the linearly independent vectors $\xi, a\xi, a^2\xi$. 
An easy computation shows that the restriction of $[a,\xi\tnz f+a\xi\tnz g]$ to this subspace is invertible. 
Hence it has three nonzero eigenvalues (counted by their multiplicity) whose sum is zero. 
Therefore $\s([a,\xi\tnz f+a\xi\tnz g])$ is not equal to $-\s([a,\xi\tnz f+a\xi\tnz g])$.
\end{proof}

\begin{corollary}
Let $a\in M_n(\CC)$.  Then there exists a nonscalar $b\in M_n(\CC)$ such that $b\not\in a+\CC 1$ and $\s([b,x])\subseteq \s([a,x])$ for every $x\in M_n(\CC)$ if and only if $a$ is algebraic of degree 2. In this case $b\in -a+\CC 1$.
\end{corollary}

\begin{proof}
Apply Lemma \ref{alg2}, Theorem \ref{rank1}, and Lemma \ref{alg3}.
\end{proof}

We will need Theorem \ref{rank1} in Section 4. Its first application, however, concerns 
derivations with the property that their
values have a finite spectrum. Such derivations have been studied in a series of papers, started in \cite{BrSe0} and ended in \cite{BoSe}.

\begin{theorem}\label{majhen}
Let $d,g$ be derivations of a semisimple Banach algebra $B$. Suppose $\s(d(x))$ is finite for every $x\in B$. If $\s(g(x))\subseteq \s(d(x))$ for every $x\in B$, then there exist derivations $d_0,d_1,d_2$ of $B$  such that  $d=d_0+d_1+d_2$, $g=d_1-d_2$, and $d_i(B)Bd_j(B)=0$ for $i\neq j$.
\end{theorem}

\begin{proof}
First we give an extraction from the aforementioned papers. It consists of the main results together with some details  that are apparently not explicitly formulated in any of the papers from the series,   but are evident from the proofs.

By \cite[Theorem 2.4]{BoMa} there exist $a,b\in\mathrm{soc}( B)$ such that $d=[a,-]$ and $g=[b,-]$.
Accordingly, each of $d(B)$ and $g(B)$ is contained in all  but finitely many primitive ideals of $B$ \cite[Proposition 2.2]{BrSe0}. 
If $d=0$, then there are no such primitive ideals for $d(B)$.  However, in this case $g$ has only quasinilpotent values and hence it is $0$ (see, e.g., \cite{TS}). We may therefore assume that $d\ne 0$. On the other hand, $g\ne 0$ can be assumed without loss of generality. 
Let $P_1,\ldots,P_m$ be the only primitive ideals such that $d(B)\not\subseteq P_i$, $i=1,\ldots,m$, and similarly, 
let $Q_1,\ldots,Q_n$ be the only primitive ideals such that $g(B)\not\subseteq Q_j$, $j=1,\ldots,n$. As noticed in the proof of \cite[Theorem 2.1]{BoSe}, we have $P_i\not\subseteq P_{i'}$ for all $i\ne i'$ and  $Q_j\not\subseteq Q_{j'}$ for all $j\ne j'$. 
  Therefore the proof of 
 \cite[Theorem 2.4]{BoMa} shows that there exist $a_1,\ldots,a_m\in B$ and $b_1,\ldots,b_n\in B$ such that
 \begin{itemize}
\item $a= a_1+\ldots+a_m$ and $b=b_1+\ldots+b_n$.
\item  $d(x)+P_i = [a_i,x]+ P_i$ and  $g(x)+Q_j = [b_j,x]+ Q_j$  for all $x\in B$.
\item $a_i+P_i\in \soc(B/P_i)$ and $b_j+Q_j\in \soc(B/Q_j)$. 
\item $a_i\in \bigcap_{P\neq P_i}P$  and $b_j\in \bigcap_{P\neq Q_j}P$.
\end{itemize}
(Here, the intersection runs over primitive ideals $P$ of $B$.) Note that each $a_i\ne 0$ and  each $b_j\ne 0$. Therefore $\bigcap_{P\neq P_i}P$ and
$\bigcap_{P\neq Q_j}P$ are nonzero ideals. Since $\Bigl(\bigcap_{P\neq P_i}P\Bigr) \bigcap P_i=0$ and $\Bigl(\bigcap_{P\neq Q_j}P\Bigr) \bigcap Q_j=0$ by the semisimplicity of $B$, it follows that
$\bigcap_{P\neq P_i}P\not\subseteq P_i$ and $\bigcap_{P\neq Q_j}P\not\subseteq Q_j$.

We claim that  
$\{Q_1,\ldots,Q_n\}\subseteq \{P_1,\ldots,P_m\}$. Suppose that, say, $Q_1$ is none of the ideals $P_i$. Thus, $d(B)\subseteq Q_1$ and $g(B)\not\subseteq Q_1$.
 Set $I = \bigcap_{P\neq Q_1} P$ and take  $x\in I$. In particular, $x$ is contained in every $P_i$, hence $d(x)$ is contained in every $P_i$, and so $d(x)$ is actually contained in every primitive ideal of $B$. Therefore $d(x)=0$, and, consequently,
$\sigma(g(x))=\{0\}$. Thus, the restriction of $g=[b,-]$ to   $I$ is a continuous derivation of $I$ with quasinilpotent values. As $I$ is a closed ideal of a semisimple Banach algebra, it follows that  $g(I)=0$ \cite{TS}. Accordingly,  for all $x\in B$ and $u\in I$ we have $g(x)u = g(xu) - xg(u)=0\in Q_1$. Since $Q_1$ is, in particular, a prime ideal, and since $g(x)\notin Q_1$ for some $x\in B$, we must have $I\subseteq Q_1$. However, at the end of the preceding paragraph we have shown that this is not true. 
Our claim is thus proved.  Therefore $n\le m$ and we may assume that 
 $$
 Q_1=P_1,\,Q_2=P_2,\ldots,\, Q_n = P_n.
 $$
 
 Recall that $\s(y)=\bigcup_{P}\s(y+P)$ for every $y\in B$ \cite[Theorem 2.2.9]{Ric}.
 Let $i\le n$, pick $x\in \bigcap_{P\neq P_i}P$, and take $g(x)$ for $y$. Then we obtain
   $\s(g(x))\cup\{0\}
  =\s(g(x)+P_i)\cup\{0\}=\s([b_i,x]+P_i)\cup\{0\}$. Similarly, 
  $\s(d(x))\cup\{0\}=\s([a_i,x]+P_i)\cup\{0\}.$ Using the assumption  of the theorem we thus have
  $$
  \s([b_i,x]+P_i)\subseteq \s([b_i,x]+P_i)\cup\{0\}\subseteq \s([a_i,x]+P_i)\cup\{0\}
  $$
for all $x\in   \bigcap_{P\neq P_i}P$. 
  Since $\bigcap_{P\neq P_i}P$ is an ideal which is not contained in $P_i$,  $\bigl(\bigcap_{P\neq P_i}P+P_i\bigr)/P_i$ is a nonzero ideal of $B/P_i$. 
  It is well-known that the socle of a primitive algebra  is contained in every other nonzero ideal. Therefore 
    $\s([b_i+P_i,y])\subseteq \s([a_i+P_i,y])\cup\{0\}$
  holds for all $y\in {\rm soc}(B/P_i)$.  
This enables us  
to apply Theorem \ref{rank1}. Hence we conclude that $b_i+P_i=\lambda_ia_i+\mu_i 1+P_i$ for some $\lambda_i\in\{-1,1\}$ and $\mu_i\in \CC$. Note that the case $\lambda_i=0$ can not occur for $g(B)\not\subseteq P_i$.  Now define $\tilde{a}_1$ as the sum of all $a_i$ such that $i\le n$ and $\lambda_i =1$,  $\tilde{a}_2$ as the sum of all $a_i$ such that $i\le n$ and $\lambda_i =-1$,
and  $\tilde{a}_0$ as the sum of all $a_i$ such that $i> n$ (the sum over the empty set of indices should be read as $0$). 
Setting   $d_0=[\tilde{a}_0,-]$, $d_1=[\tilde{a}_1,-]$, and $d_2=[\tilde{a}_2,-]$ we have $d=d_0+d_1+d_2$ and $g=d_1-d_2$.
Note that for any pair of different indices $k$ and $l$ we have $[a_k,B]B[a_l,B]\subseteq  \bigr(\bigcap_{P\neq P_k}P\bigl)\bigcap \bigr(\bigcap_{P\neq P_l}P\bigl) =0$.
This clearly implies that $d_i(B)Bd_j(B)=0$ if $i\neq j$.
\end{proof}

So far we have relied heavily on finite rank operators. In general Banach algebras the spectrum of a value of a derivation may not be so easily tractable. The next lemma reduces the treatment of the spectral inclusion condition \eqref{ena} to another problem which may be of independent interest. It will play a fundamental role in the next section and in the first result of the last section.

\begin{lemma}\label{Bo}
Let $B$ be a semisimple Banach algebra and $a,b\in B$. If $\s([b,x])\subseteq \s([a,x])\cup\{0\}$ for all $x\in B$, then for all $y,z \in B$,  $yz=yaz=0$ implies $ybz=0$.
\end{lemma}

\begin{proof}
From the assumptions $yz=0$ and $yaz=0$ we find by a short calculation that $[a,zxy]^3=0$ for all 
$x\in B$. Consequently, by the hypothesis of the lemma,  $\s([b,zxy])=\{0\}$ for all $x\in B$. Assume that $ybz\ne0$ and, to 
obtain a contradiction, take an irreducible representation $\pi$ of $B$ on a Banach space $X$ such that $\pi(ybz)\neq 0$.  Choose $\xi\in X$ with $\pi(ybz)\xi\neq 0$.  By irreducibility there exists $u\in B$ such that $\pi(u)\pi(ybz)\xi=\xi$. Then $\pi([b,zuy])\eta=-\eta$, where $\eta=\pi(z)\xi$. Hence  we have $-1\in \s([b,zuy])$, a contradiction. 
\end{proof}

\section{Results on von Neumann algebras}

In this section we consider the spectral inclusion condition \eqref{ena} in von Neumann algebras. As derivations are
automatically inner on these algebras (see, e.g., \cite[Exercise 8.7.55]{KR}), we assume, throughout, that  $d=[a,-]$ and $g
 =[b,-]$.

For factors the desired conclusion, which is the same as for primitive Banach
algebras with nonzero socle, follows easily from Lemma \ref{Bo} and the result on
reflexivity from \cite{Mag}. But first we need a technical lemma.

\begin{lemma}\label{-1,1}
Let $B$ be a von Neumann algebra and $I$  be a closed ideal in $B$. If an element $a\in B$ with $a+I\not \in Z( B/I)$ satisfies $\lambda\s([a,x]+I)\subseteq \s([a,x]+I)\cup\{0\}$ for all $x\in B$ and 
for some $\lambda\in \CC$, then $\lambda\in \{-1,0,1\}$.
\end{lemma}

\begin{proof}
Since $a+I$ is not central in $B/I$, we can find a projection $p\in B$ such that $(1-p)ap\not \in I$. 
Let $B$ act on a Hilbert space $H=pH+(1-p)H$.
According to this decomposition we can represent every $a\in B$ as a $2\times 2$ matrix 
$$\left( \begin{array}{cc}
a_1 & a_2\\
a_3 & a_4
\end{array} \right),$$
where $a_1=pap$, $a_2=pa(1-p)$ and so on. If we choose $x$ in $pB(1-p)$ 
so that $x$ is represented by the matrix which has an element $x_2$ on the position $(1,2)$ and zeros 
elsewhere,
then a short computation shows that
$$
[a,x]=\left( \begin{array}{cc}
-x_2a_3 & a_1x_2-x_2a_4\\
0  & a_3x_2
\end{array} \right).$$
Since $a_3=(1-p)ap\notin I$, it follows that $a_3a_3^*\notin I$ and
there exists a closed subset in $\s(a_3a_3^*)$ that does not contain 0 such that its characteristic function
$\chi$ satisfies $\chi(a_3a_3^*)
\not\in I$.  
Take a function $g$ that satisfies $\chi(t)=g(t)t$ for every $t\in \s(a_3a_3^*)$ and $g(0)=0$. 
If  we choose $x_2=a_3^*g(a_3a_3^*)$, then $[a,x]$ is of the form
$$[a,x]=\left(\begin{array}{cc}
-q_1&y\\
0&q_2\end{array}\right),$$
where $q_1=\chi(a_3^*a_3)$ and $q_2=\chi(a_3a_3^*)$.
(We have used the well-known fact that $g(a_3a_3^*)a_3=a_3g(a_3^*a_3)$ which follows by approximating
$g$ by polynomials.) Since $q_1$ and $q_2$ are projections and not contained in $I$, we now see that 
 $\{-1,1\} \subseteq  \s([a,x]+I)\subseteq \{-1,0,1\}$, from which the lemma is evident. 
\end{proof}

\begin{proposition}
Let $B$ be a factor and let $a,b\in B$  satisfy $\s([b,x])\subseteq \s([a,x])\cup\{0\}$ for every $x\in B$. 
Then $b=\lambda a+\mu$ for some $\lambda\in \{-1,0,1\}$ and some $\mu\in \CC$.
\end{proposition}

\begin{proof}
The case $B=B(H)$ has already been handled by
 Theorem \ref{rank1}.
Assume that $B\neq B(H)$ and $b\not \in  \spn\{1,a\}$. 
According to \cite[Theorem 1.1]{Mag}, every finite dimensional subspace, in particular  $\spn\{1,a\}$, in a factor different from $B(H)$ is reflexive. 
This means that there exist $y,z\in B$ such that $yz=0$ and $yaz=0$ but $ybz\neq 0$, contradicting Lemma  \ref{Bo}.
 Therefore $b\in \spn \{1,a\}$. 
Now Lemma \ref{-1,1} with $I=0$ yields the desired conclusion. 
\end{proof}

We now proceed to general von Neumann algebras. Let us introduce some notation and list some standard  results that will be needed in the sequel.

Denote by $X$ the character space of $Z(B)$. 
Let $Bt$ be the closed ideal in $B$ generated by $t\in X$. 
We write $B(t)$ for the quotient algebra $B/Bt$ and $x(t)$ for the coset $x+Bt\in B/Bt$. 
The function $t\mapsto \|x(t)\|$ is continuous for every $x\in B$  and the map  $x\mapsto (x(t))_{t\in X}$ from $B$ to $\Pi_{t\in X} B(t)$ is injective and hence an isometric embedding, in particular $\|x\|=\sup_{t\in X}\|x(t)\|$ (see  \cite{Glimm} for proofs). By \cite[Theorem 4.7]{Halp}, $B(t)$ is primitive for every $t\in X$.

The spectrum of elements relative to some subalgebra $A$ of $B$ will be denoted  by $\s_A(\,.\,)$.

\begin{lemma}\label{kvoc}
Let $B$ be a von Neumann algebra, $c$ an element in $B$, and let $t$ be an element in the character space $X$ of $Z(B)$. 
If $\emph{P}_t$ is the set of all projections that correspond to the characteristic functions of those clopen sets that contain $t$,
then $\s(c(t))=\bigcap_{p\in \emph{P}_t}\s_{pB}(pc)$.
\end{lemma}

\begin{proof}
Since $B(t)$ is a quotient of $pB$ for every $p\in \emph{P}_t$ , $\s(c(t))\subseteq\bigcap_{p\in \emph{P}_t}\s_{pB}(pc)$.

For the reverse inclusion assume that $c(t)$ is invertible. It suffices to show that there exists a clopen set $U\subseteq X$ which contains $t$
and that $c(s)$ is invertible for every $s\in U$. Then $pc$ is invertible in $pB$ where $p$ is the projection 
that corresponds to the characteristic function of $U$.

Consider the polar decomposition $c=u|c|$ of $c$, hence $c(s)=u(s)|c(s)|$ for all $s\in X$. 
Since $c(t)$ is invertible, $u(t)$ is unitary. The continuous functions
$s\mapsto \|u(s)u^*(s)-1\|$ and $s\mapsto \|u^*(s)u(s)-1\|$ equal zero at $s=t$. Hence, these 
functions are less than $1$ on some neighborhood $V$ of $t$ in $X$.
Since $u(s)$ is a partial isometry for every $s\in X$, it follows that $u(s)$ must be invertible
(thus unitary) for all $s$ in $V$. 
Hence we  may assume that $c\geq 0$ on $V$. As $c(t)$ is invertible, $c(t)\geq m 1$ for some $m\in \R^+$. 
We only need to show that $c(s)>\frac{m}{2}1$ for all $s$ is some neighborhood of $t$. 
Suppose the contrary that there exists a net $\{s_j\}_{j\in J}$ converging to $t$ such that $\s(c(s_j))$ contains some $\lambda_j<\frac{m}{2}$ for every $j\in J$.
Let $f:\R\to \R$ be the continuous map defined by 
$$
f(x)=\left\{ \begin{array}{ll}
1 & \textrm{if $x<\frac{m}{2}$},\\
2-\frac{2}{m}x & \textrm{if $\frac{m}{2}\leq x\leq m$}, \\
0 &  \textrm{if $x>m$}.
\end{array}\right.
$$
Then $f(c(t))=0$ and $\|f(c(s_j))\|=1$ for every $j\in J$.
As $f(c)(s)=f(c(s))$, we have $f(c)(t)=0$ and $\|f(c)(s_j)\|=1$ for every $j\in J$.
However, this contradicts the continuity of the map $s\mapsto \|f(c)(s)\|$.
\end{proof}

\begin{theorem}
Let $B$ be a  von Neumann algebra and let $a,b\in B$. 
If $\s([b,x])\subseteq\s([a,x])\cup\{0\}$ for every $x\in B$, then $b=p_1 a-p_2 a+z$ for some orthogonal 
central projections $p_1,p_2$ and some $z\in Z(B)$.
\end{theorem}

\begin{proof}
There exist central orthogonal projections $z_1,z_2\in B$ with $z_1+z_2=1$  
such that $z_1B$ is of Type $I$ while $z_2B$
does not contain central portions of Type $I$. 
We have the inclusions $\s([z_ib,x])\subseteq\s([z_ia,x])\cup\{0\}$ for every $x\in z_iB$, $i=1,2$, therefore $\s_{z_iB}([z_ib,x])\subseteq\s_{z_iB}([z_ia,x])\cup\{0\}$. 
Hence the proof will be divided into two cases, the one where $B$ is of Type $I$ and the one  where $B$ does not contain central portions of Type $I$.

\textit{Case 1.}
Let $B$ be a von Neumann algebra of Type $I$.  It suffices to show that the assertion of the theorem holds for $B=C(X)\overline{\tnz} B(H)$, the von Neumann algebra of continuous functions from a Stonean space $X$ with values in $B(H)$, equipped with a weak operator topology, for a Hilbert space $H$. Indeed, $B$ is a direct sum of such algebras.

We take $\uu,\eta\in H$ with $\ls\xi,\eta\rs=0$.
Choose $\epsilon>0$, $t_0\in X$ and let $U$ be a clopen neighborhood of $t_0$ with the property $|\ls(a(t)-a(t_0))\xi,\eta\rs|<\epsilon$ and $|\ls (b(t)-b(t_0))\xi,\eta\rs|<\epsilon$ for every $t\in U$.
We calculate the spectra of $[a,(\xi\tnz \eta^*)_U]$ and $[b,(\xi\tnz \eta^*)_U]$, where $(\xi\tnz \eta^*)_U$ denotes the function $t\mapsto \chi_U(t)\xi\tnz \eta^*$ for every $t\in X$.
We have 
$$\s([a,(\xi\tnz \eta^*)_U])\cup \{0\}=\cup_{t\in U}\s([a(t),\xi\tnz \eta^*])\cup\{0\}$$ 
and 
$$\s([b,(\xi\tnz \eta^*)_U])\cup\{0\}=\cup_{t\in U}\s([b(t),\xi\tnz \eta^*])\cup\{0\}.$$ 
From the second paragraph of the  proof of Theorem \ref{rank1} we see that $\s([a(t),\xi\tnz\eta^*])\subseteq \{0,-\ls a\xi, \eta\rs,\ls a\xi, \eta\rs\}$ and $\s([a(t),\xi\tnz\eta^*])=\{0\}$ if and only if $\ls a(t)\xi,\eta\rs=0$.   Thus,
$$\s([a,(\xi\tnz \eta^*)_U])\subseteq\cup_{t\in U}\{0,-\ls a(t)\xi, \eta\rs,\ls a(t)\xi, \eta\rs\}$$ 
and by choice of $U$ we  have 
$|\ls a(t)\xi, \eta\rs -\ls a(t_0)\xi, \eta\rs|<\epsilon $.
The same conclusions hold if we replace $a$ by $b$. 
Since $\s([b,(\xi\tnz \eta^*)_U])\subseteq \s([a,(\xi\tnz \eta^*)_U])\cup\{0\}$, it follows that $\ls b(t_0)\xi,\eta\rs=0$ or $|\ls b(t_0)\xi,\eta\rs-\ls a(t_0)\xi, \eta\rs|<\epsilon$ or $|\ls b(t_0)\xi,\eta\rs+\ls a(t_0)\xi, \eta\rs|<\epsilon$ for every $\epsilon>0$. Consequently, we have $\ls b(t_0)\xi,\eta\rs=0$ or $\ls b(t_0)\xi,\eta\rs=\pm\ls a(t_0)\xi,\eta\rs$. Following the proof of Theorem \ref{rank1} we may conclude that $b(t_0)\in \CC 1$ or $b(t_0)\pm a(t_0)\in \CC 1$ for every $t_0\in X$.

Therefore, the union of the closed sets $F_0=\{t\in X:\; b(t)\in \CC 1\}$, $F_1=\{t\in X:\; b(t)-a(t)\in \CC 1\}$ and $F_2=\{t\in X:\; b(t)+a(t)\in \CC 1\}$ equals $X$. Complements of these sets are open and the 
interiors $G_0,G_1,G_2$ of $F_0,F_1,F_2$, respectively, are clopen. 
As $G_0^\co$ is the closure of $F_0^\co$, we have $G_0^\co\subseteq F_1\cup F_2$.
The sets $G_0$, $G_0^\co$ then divide $X$ in the disjoint union of two
clopen sets, contained in $F_0$, $F_1\cup F_2$, respectively. 
Similarly, we divide $G_0^\co$ in the union of two disjoint clopen sets 
$H_1\subseteq F_1$, $H_1^\co\subseteq F_2$. 
The characteristic functions of $G_0,H_1,H_1^\co$ yield central projections 
$p_0,p_1,p_2$ with sum $1$. Moreover, the elements $z_0=p_0b$, $z_1=p_1(b-a)$, $z_2=p_2(b+a)$ are central.
The result is $b=(p_0+p_1+p_2)b=p_1a-p_2a+z_0+z_1+z_2$.

\textit{Case 2.}
Assume now that $B$ does not contain central portions of Type $I$. 
Let us examine the linear independence of $1,a,b$ in the quotient spaces $B(t)$ for $t\in X$.
If the elements $1,a(t_0),b(t_0)$ are linearly independent for some $t_0\in X$, 
then there exists a neighborhood $U$ of $t_0$ such that $1, a(t),b(t)$ are linearly independent 
for every $t\in U$. 
(In order to prove this, we consider the  continuous map $f:X\times S\to \R$  
defined by $f(t,\alpha, \beta,\gamma)\mapsto \|\alpha 1+\beta a(t)+\gamma b(t)\|$, 
where $S$ is the  unit sphere in $\CC^3$.
As $1, a(t_0), b(t_0)$ are linearly independent, $f(t_0,\alpha,\beta,\gamma)>m$ 
for some $m>0$ and for all $(\alpha,\beta,\gamma)\in S$. 
Since $S$ is compact we can find a neighborhood $U$ of $t_0\in X$ 
such that $f(t,\alpha,\beta,\gamma)>\frac{m}{2}$ for all $t\in U$, $(\alpha,\beta,\gamma)\in S$.)

Replacing $U$ by its appropriate subset, if necessary, we may assume that $U$ is clopen.
Its characteristic function is continuous on $X$ and corresponds to some projection $p\in Z(B)$.
Thus $1,a(t),b(t)$ are linearly independent for every $t$
in the character space $X_p$ of $Z_p=Z(pB)$. 
From the hypothesis and since $p$ is central, we have $\s([pb,x])\subseteq \s([pa,x])\cup\{0\}$ 
for every $x\in pB$. 

Consider the map $f':Z_p^2\to pB$, 
defined by $(z_1,z_2)\mapsto z_1+z_2 pa$. 
Since $1$ and $a(t)$ are linearly independent for all $t\in X_p$,  there exists $m'>0$ such that
$\|\alpha+\beta pa(t)\|^2\geq m' (|\alpha|^2+|\beta|^2)$ for every $\alpha, \beta\in \CC$.
Therefore, $\|z_1(t)+z_2(t)pa(t)\|^2\geq m' (|z_1(t)|^2+|z_2(t)|^2)$ for all $z_1,z_2\in Z_p$.
Taking the supremum over $t\in X_p$ gives $\|z_1+z_2pa\|^2\geq m' \max\{\|z_1\|^2,\|z_2\|^2\}$. 
Thus the map $f'$ is bounded from below, hence its range is norm-closed. 
Since the range of any weak* continuous linear map is norm closed if and only if 
it is weak* closed (see, e.g., \cite[Chapter VI, Theorem 1.10]{Con}), $Z_p+Z_pa$ is weak* closed.  
Since $pB$ does not contain central portions of Type I, 
$Z_p+Z_p a$ is reflexive by \cite[Theorem 1.1]{Mag}. Therefore there exist $y,z\in pB$ such that $yz=0$, $ypaz=0$ but $ypbz\neq 0$, 
which contradicts Lemma \ref{Bo}.
Thus, $b(t)\in\spn\{ 1,a(t)\}$ 
for every $t\in X$ for which $1,a(t)$ are linearly independent.
Hence in this case there exist $\lambda (t),\mu(t)\in \CC$ such that $b(t)=\lambda(t) a(t)+\mu(t)$.
Using Lemma \ref{kvoc} we deduce the inclusion $\s([pb(t),px(t)])\subseteq \s([pa(t),px(t)])\cup\{0\}$ 
for all $x\in B$; 
then by Lemma \ref{-1,1} we conclude that $\lambda(t)\in \{-1,0,1\}$. 
In the remaining case, when $a(t)$ is a scalar for some $t$, $b(t)$ must also be a scalar by Lemma \ref{kvoc} and \cite[Theorem 5.2.1]{Aup}.  

Therefore, the union of the closed sets $F_0=\{t\in X:\; b(t)\in \CC 1\}$, 
$F_1=\{t\in X:\; b(t)-a(t)\in \CC 1\}$ and $F_2=\{t\in X:\; b(t)+a(t)\in \CC 1\}$ equals $X$. 
The proof can now be completed by  the same argument as at the end of Case 1.
\end{proof}

\section{Results on $C^*$-algebras}

In this section we will consider inner derivations of $C^*$-algebras. Our first result is an easy consequence of Lemma \ref{Bo} and a deeper result from \cite{ABEV}.

\begin{theorem}
Let $B$ be a unital $C^*$-algebra and $a,b\in B$ with $a$ normal. If $\s([b,x])\subseteq \s([a,x])$ for all $x\in B$, then $b\in\{a\}''$, the bicommutant of $\{a\}$ in $B$.
\end{theorem}

\begin{proof}
Let $A=\{a\}'$. Define $\phi:A\times A\to B$ by $\phi(y,z)=ybz$. For all $u,v\in A$, $uv=0$ implies $uav=0$. 
According to Lemma \ref{Bo} this further gives $\phi(u,v)=0$. 
Then $\phi(xy,z)=\phi(x,yz)$ for all $x,y,z\in A$ \cite[Theorem 2.11 and Example 2, p. 137]{ABEV}. 
Setting $x=z=1$ we get $yb=by$.
\end{proof}

From now on we will consider the norm inequality condition \eqref{dva}, which, as observed in the introduction, follows immediately from the spectral inclusion condition \eqref{ena} if $a$ and $b$ are selfadjoint. 
When can \eqref{dva} occur? 
For instance, if $b=a^n$, then we see from $[b,x]=a^{n-1}[a,x]+a^{n-2}[a,x]a+\cdots +[a,x]a^{n-1}$  
that \eqref{dva} holds with $M=n\|a\|^{n-1}$. Consequently, \eqref{dva} is fulfilled whenever $b$ is 
a polynomial in $a$. In the next theorem we will show directly that \eqref{dva} implies that $b$ is a 
Lipschitz function $f$ of $a$, provided that $B$ is a prime C$^*$-algebra and $a$ is a normal element. 
This is perhaps not the best possible conclusion, however, a complete description of the properties
of the appropriate functions $f$ could be too difficult. In \cite{JW} Johnson and Williams 
considered the  special case where $B=B(H)$. 
By \cite[Corollary 3.7]{JW}, in this case the condition (N) is equivalent to the requirement
that the range of $[b,x]$ is contained in the range of  $ [a,x] $ for every $x\in B(H)$. 
The description of appropriate  functions $f$ in the case $B=B(H)$ in   \cite[Theorems 3.6 and 4.1]{JW}
is quite entangled. In a general
C$^*$-algebra $B$ the range inclusion ${\rm im}\,[b,x]\subseteq{\rm im}\,[a,x]$ for all $x\in B$ implies the condition
(N) by \cite[Theorem 6.5]{KS1}.

\begin{theorem}\label{nenorm}
Let $B$ be a  prime $C^*$-algebra and let $a,b\in B$ satisfy $\|[b,x]\|\leq M\|[a,x]\|$ for all 
selfadjoint $x\in B$ and some $M>0$. If $a$ is normal, then $b=f(a)$  where $f$ is Lipschitz (with a Lipschitz constant $M$) on the spectrum of $a$.
\end{theorem}

\begin{proof}
We can assume $a\not\in \CC 1$ without loss of generality.

Our assumption implies that $a$ and $b$ commute. Since $a$ is normal, $a$ commutes also with $b^*$ by
the Putnam-Fuglede theorem, and then  the condition (N) implies that $b$ 
is normal.
Denote by $A$ the  $C^*$-algebra generated by $a$ and $b$. The Gelfand transformation is an 
isomorphism between $A$ and $C(\Omega)$, the algebra of continuous functions on the character space $\Omega$  of $A$, which can be identified with a compact subset $K$ of $\CC^2$ 
via the homeomorphism $\psi:\chi\mapsto (\chi(a),\chi(b))$.
Let $\A$ and $\B$ denote the Gelfand transforms of $a$ and $b$, regarded as functions on $K$.
(These are just the restrictions to $K$ of the two coordinate projections $\mathbb{C}^2\to\mathbb{C}$.)

We can divide $\mathbb{C}^2=\mathbb{R}^4$ into a grid of small closed cubes with sides parallel to the 
coordinate axes such that the intersection of any two cubes is either empty or a common face. Then it is 
not hard to see that there exists $p\in\mathbb{N}$ such that each
cube intersects at most $p$  other cubes ($p=3^4-1$). By taking slightly larger open cubes we can cover the
compact set $K$ by a finite family $\{P_i\}_{i=1}^n$ of such cubes, so that each intersects at most
$p$ other cubes; moreover, for a given $\epsilon>0$, by the uniform continuity we may assume that the
cubes $P_i$ are so small that $|\A(t)-\A(t')|<\epsilon$ and $|\B(t)-\B(t')|<\epsilon$ whenever
$t,t'\in K$ are such that $t\in P_i$ and $t'\in P_j$ for some $i,j$ with $P_i\cap P_j\ne\emptyset$.
  
 Let $V_i=P_i\cap K$ and choose a partition of unity $\{\hat{\phi_i}\}_{i=1}^n$ subordinate to
 the covering $\{V_i\}_{i=1}^n$ of $K$.
 Then 
 $$\|\A-\sum_{i=1}^n \A(t_i)\hat{\phi_i}\|<\e,\  \mbox{and}\  \|\B-\sum_{i=1}^n \B(t_i)\hat{\phi_i}\|
 <\e\ \mbox{for arbitrary}\ t_i\in V_i.$$
 Note also that $\hat{\phi_j}\sum_{i:V_i\cap V_j\neq \emptyset}\hat{\phi_i}=\hat{\phi_j}$ for all $ 1\leq j \leq n$.
 Moreover, if $V_i\cap V_j\neq \emptyset$ then $|\A(t_i)-\A(t_j)|<\e$ for arbitrary $t_i\in V_i, t_j\in V_j$.
   
 Choose $1\leq j,k\leq n$ such that $V_j\cap V_k=\emptyset$. Since $B$ is prime, there exists 
 $y\in B$ such that $\|\phi_jy\phi_k\|=1$. As $\phi_j\phi_k=\phi_j\phi_k=0$, we have $\|\phi_jy\phi_k\pm\phi_ky^*\phi_j\|= 2$.
   Set $x=\phi_jy\phi_k+\phi_ky^*\phi_j$. Take arbitrary $t_i\in V_i$. Let $L(j)$ denote  the set of all $i\in \{1,\ldots, n\}$ such that $V_i\cap V_j\neq \emptyset$. For $p,q\in \R$ we write $p\approx_\e q$ if $|p-q|<\e$.
    According to the above observations we can estimate 
  \begin{eqnarray*}
  \|[a,x]\|&\approx_{4\e}& \left\|\left[\sum\nolimits_{i=1}^n\A(t_i)\phi_i,x\right]\right	\| \\
  &=& \left\|\sum\nolimits_{L(j)}\A(t_i)(\phi_i\phi_jy\phi_k-\phi_ky^*\phi_j\phi_i)
    -\sum\nolimits_{L(k)}\A(t_i)(\phi_jy\phi_k\phi_i-\phi_i\phi_ky^*\phi_j)\right\| \\
  &\approx _{4p\e}&
    \left\|\A(t_j)\phi_j\left(\sum\nolimits_{L(j)}\phi_i\right)y\phi_k-\A(t_j)\phi_ky^*\phi_j\left(\sum\nolimits_{L(j)}\phi_i\right)\right.\\
  & & \left.-\A(t_k)\phi_jy\phi_k\left(\sum\nolimits_{L(k)}\phi_i\right)+\A(t_k)\phi_k\left(\sum\nolimits_{L(k)}\phi_i\right)y^*\phi_j\right\| \\
  &=&\|\A(t_j)(\phi_jy\phi_k-\phi_ky^*\phi_j)-\A(t_k)(\phi_jy\phi_k-\phi_ky^*\phi_j)\| \\
  &=&2 |\A(t_j)-\A(t_k)|.
  \end{eqnarray*} 
   In the same manner we show that $\|[b,x]\|\approx_{4(p+1)\e}2|\B(t_j)-\B(t_k)|$.
   Using the condition (N) it follows now that $0\leq M\|[a,x]\|-\|[b,x]\|\approx_{8(p+1)\e} 2
   (M|\A(t_j)-\A(t_k)|-|\B(t_j)-\B(t_k)|)$. 
   Since $\e>0$ and $t_i\in V_i$ were arbitrary, we conclude that 
   \begin{equation*}\label{L}|\B(t)-\B(t')|\leq M|\A(t)-\A(t')|\ \mbox{for all}\ t, t'\in K.\end{equation*}
   From this we see in particular that $\A(t)=\A(t')$ implies $\B(t)=\B(t')$, hence $b$ is a function of
   $a$, say $b=f(a)$. The above inequality  means that this function $f$ is Lipschitz.
\end{proof}

   \begin{corollary}\label{Tn}
    Let $B$ be a prime  $C^*$-algebra. The following conditions are equivalent for normal elements $a, b\in B$:
    \begin{enumerate}
    \item $\|[a,x]\|=\|[b,x]\|$ for every selfadjoint $x\in B$.
    \item $b=\lambda a+\mu 1$  or $b=\lambda a^*+\mu 1$ for some $\lambda, \mu\in \CC$ with $|\lambda|=1$.
    \end{enumerate}
  \end{corollary}
  
  \begin{proof}
By replacing $a$ with $a-t_01$ for some $t_0\in \s(a)$ we may assume that $0\in \s(a)$. 
From Theorem \ref{nenorm} we obtain $b=f(a)$ for a Lipschitz function $f$ on $\s(a)$ and $a=g(b)$ for a Lipschitz function $g$ on $\s(b)$, with both $f$ and $g$ having Lipschitz constants 1. Then $g\circ f=id_{\s(a)}$ yields $|f(t)-f(t')|=|t-t'|$ for all $t,t'\in \s(a)$.
 Replacing $f$ by $f-f(0)$ we may assume that $f(0)=0$.
 It is easy to see that then $f$ must take one of the following forms: $f(t)=e^{i\theta}t$ or 
 $f(t)=e^{i\theta}\overline t$ for some $\theta\in [0,2\pi)$. This establishes (2).

The converse is trivial.
  \end{proof}

  \begin{theorem}\label{T1}
  Let $B$ be a  $C^*$-algebra on some Hilbert space and $\BB$ be its weak$^*$ closure. The following conditions are equivalent for selfadjoint elements $a, b\in B$:
  \begin{enumerate}
  \item $\rho([a,x])=\rho([b,x])$ for every $x\in B$.
  \item $\|[a,x]\|=\|[b,x]\|$ for every selfadjoint $x\in B$.
  \item For every primitive ideal $P$ of $B$ we have that $a+b+P\in \CC 1+P$ or $a-b+P\in \CC 1+P$.
  \item $b=ca+z$ for some $c,z\in Z(\BB)$ with $z=z^*, c=c^*, c^2=1$.
  \end{enumerate}
  \end{theorem}
  
  \begin{proof}
Take a selfadjoint $x\in B$. Then $[a,x]$, $[b,x]$ are anti-selfadjoint and $\|[a,x]\|=\rho([a,x])$, $\rho([b,x])=\|[b,x]\|$. Thus,  (1) implies (2).
  
Assume that (2) holds. 
We will first observe that a passage to the quotient  $B/I$ for an arbitrary ideal $I$ in $B$ 
preserves the condition (2).
 This observation follows from the existence of a quasicentral approximate unit in $I$ (see, e.g., \cite[Exercise 10.5.6]{KR}).
 This is an increasing net $\{e_\lambda\}_{\lambda\in \Lambda}$ of positive elements in $I$ which
 is an approximate unit for $I$ and satisfies $\lim_\lambda\|e_\lambda z-ze_\lambda\|= 0$ for every $z\in B$.
 Since $\|y+I\|=\lim _\lambda \|(1-e_\lambda)y\|$ for all $y\in B$ (see, e.g., \cite[Exercise 4.6.60]{KR}), we have 
\begin{align*}
\|[a,x]+I\| |=\lim_\lambda\|(1-e_\lambda)[a,x]\|=\lim_\lambda\|[a,(1-e_\lambda) x]\|\\
=\lim_\lambda\|[b,(1-e_\lambda) x]\|                                                             
                      =\lim_\lambda\|(1-e_\lambda)[b,x]\|                                                     =\|[b,x]+I\|.
\end{align*}
Taking for $I$ any primitive ideal $P$, we now deduce from Corollary \ref{Tn} that $a+b+P\in \CC 1+P$ 
or $a-b+P\in \CC 1+P$.  This proves (3).

If we assume (3), then $[a+b,x]\in P$ for every $x\in B$ or $[a-b,y]\in P$ for every $y\in B$. Hence $[a-b,B]B[a+b,B]\subseteq P$ for every primitive ideal $P$.
 Since $B$ is semisimple, we obtain $[a-b,B]B[a+b,B]=0$. Write $S=[a-b,B]$, $T=[a+b,B]$. 
 It follows $S\BB T=0$. Since $\BB$ is a von Neumann algebra, 
there exists a projection $p\in Z(\BB)$ such that $ps=s$ for all $s\in S$ and $(1-p)t=t$ for all $t\in T$.
 Therefore, $(1-p)[a-b,x]=0$ and $p[a+b,x]=0$ for all $x\in B$.
 Since $p$ is central, $z_1=(1-p)(a-b)\in Z(\BB)$ and $z_2=p(a+b)\in Z(\BB)$.
 Hence, $b=(1-p)b+pb=(1-p)a-z_1+z_2-pa=(1-2p)a+z$ for an element $z=z^*\in Z(\BB)$. Taking $c=1-2p$ establishes (4). 

 The implication (4)$\Rightarrow $(1) is clear.
  \end{proof}
  
The elements $c$ and $z$ may be sometimes chosen from  $Z(B)$, but in general we cannot expect this.
\begin{example}
Take $B=C([0,1],M_2(\CC))$, the $C^*$-algebra of continuous functions with matrix values. 
Define $a$ and $b$ by $a(t)=|1-2t|J$, $b(t)=(1-2t)J$, where $J$ is an arbitrary non-scalar matrix. The two
elements $a$ and $b$ satisfy (1), therefore  $b=ca+z$ for some $c,z\in Z(\BB)$ with $z=z^*, c=c^*, c^2=1$.
If $c,z$  belonged to $B$, we would have $b(t)=c(t)a(t)+z(t)$ for every $t\in [0,1]$. Hence, $c(t)=1$ on $[0,\frac{1}{2})$ and
$c(t)=-1$ on  $(\frac{1}{2},1]$, contradicting the continuity of $c$.
\end{example}

\begin{remark}
If the primitive spectrum ${\rm Prim}(B)$ is Hausdorff, then  in Theorem \ref{T1} we can choose $z \in Z(B)$.
To see this, take selfadjoint elements $a,b\in B$ that satisfy the conditions of Theorem \ref{T1}.
Define $\U=\{P; a+b+P\not \in \CC 1+P\}$, $\V=\{P; a-b+P\not \in \CC 1+P\}$. According to the condition (3), 
$\U$ and $\V$ are disjoint sets. 
Since ${\rm Prim}(B)$ is a Hausdorff topological space, $P\mapsto \|z+P\|$ is a continuous function from ${\rm Prim}(B)$ to $\R$ for every $z\in B$ (see, e.g., \cite[Proposition 4.4.5]{Ped}). 
Hence, $\U$ and $\V$ are open sets. 
From the definition of $\U$ and the Hausdorff property of ${\rm Prim}(B)$ it follows that $a$ restricted to $\partial \U$ is a continuous scalar function.
By the Tietze extension theorem we can extend it to a continuous function on ${\rm Prim}(B)$, 
which is by the Dauns-Hofmann theorem an element $z_0\in Z(B)$. We can replace $a$ with $a-z_0$, 
and therefore assume that $a(P)=0$ for all $P\in\partial \U$. Let $\chi_\U$ be the characteristic function of $\U$, 
which is a Borel function and therefore an element of $C({\rm Prim}(B))''=Z(B)''\subseteq Z(B'')$. Define $z(P)=(1-2\chi_\U(P))a(P)+b(P)$.
This is a scalar-valued function that equals $-a(P)+b(P)$ for $P\in \U$ and $a(P)+b(P)$ for $P\in \U^\co$. 
Since $a(P)=0$ for $P\in\partial \U$, $z$ is a continuous function on ${\rm Prim}(B)$. Hence, $z\in Z(B)$. This implies the desired conclusion $b=(2\chi_\U-1)a+z$, where $2\chi_\U-1\in Z(B'')$ and $z\in Z(B)$.
\end{remark}


\begin{thebibliography}{101}

\bibitem{ABEV} J. Alaminos, M. Bre\v sar, J. Extremera and A. R. Villena,  Maps preserving zero products,  {\em Studia Math.} {\bf 193} (2009),  131--159.

\bibitem{Aup}B. Aupetit, {\em A primer on spectral theory}, Springer Verlag, 1991.


\bibitem{Dun}F.\,F. Bonsall, J. Duncan, {\em Complete normed algebras}, Springer-Verlag, 1973.


\bibitem{BoMa}N. Boudi, M. Mathieu, Commutators with finite spectrum, {\em Illinois J. Math.} {\bf 48} (2004),  687--699. 


\bibitem{BoSe}N. Boudi, P. \v Semrl,  Derivations mapping into the socle, III, {\em  Studia Math.} {\bf 197} (2010),  141--155.



\bibitem{BrMa} M. Bre\v sar, M. Mathieu, Derivations mapping into the radical, III, {\em J. Funct. Anal.} {\bf 133} (1995), 21--29.

\bibitem{BrSe0} M. Bre\v sar, P. \v Semrl, Derivations mapping into the socle, {\em Math. Proc. Camb. Phil. Soc.}  {\bf 120} (1996),  339--346. 


\bibitem{BS} M. Bre\v sar, \v S. \v Spenko, Determining elements in Banach algebras through spectral properties, preprint.




\bibitem{CKL}M.\,A. Chebotar, W.-F. Ke, P.-H. Lee,  On a Bre\v sar-\v Semrl conjecture and derivations of Banach algebras, {\em Quart J. Math.} {\bf 57} (2006),  469--478. 

\bibitem{Con}J.\, B. Conway, {\em A course in functional analysis},  Springer-Verlag, 1990.

\bibitem{Fong}C.\,K. Fong, Range inclusion for normal derivations, {\em Glasgow Math. J.} {\bf 25} (1984) 255--262.

\bibitem{Glimm} J. Glimm, A Stone-Weierstrass theorem for $C^*$-algebras, {\em Ann. Math.} {\bf 72} (1960), 216--244.

\bibitem{Halp} H. Halpern, Irreducible module homomorphisms of a von Neumann algebra into its centre, {\em Trans. Amer. Math. Soc.} {\bf 140} (1969), 195--221.

\bibitem {JW}B.\,E. Johnson, J.\,P. Williams, The range of a normal derivation, {\em Pacific J. Math.} {\bf 58} (1975), 105--122.

\bibitem{KR}R.\,V. Kadison, J.\,R. Ringrose, {Fundamentals of the theory of operator algebras, Vol. II}, Academic press, London, 1986.

\bibitem{K}  I. Kaplansky, {\em Algebraic and analytic aspects of
operator algebras}, Regional Conference Series in Mathematics 1, Amer. Math. Soc.,  1970.

\bibitem{KS1} E.
Kissin, V.\,S. Shulman, On the range inclusion of normal derivations: variations on a theme by Johnson, Williams and Fong,
{\em Proc. London Math. Soc.} {\bf  83} (2001), 176--198. 

\bibitem{Lee} T.-K. Lee, Derivations on noncommutative Banach algebras,
{\em Studia Math.} {\bf 167} (2005),  153--160. 



\bibitem{Mag} B. Magajna, 
On the distance to finite-dimensional subspaces in operator algebras, {\em J. London Math. Soc.} {\bf 47} (1993), 516--532.



\bibitem{Ped} G.\, K. Pedersen, {\em $C^*$-algebras and their automorphism groups}, Academic press, Inc., 1979.


\bibitem{Ric}C.\,E. Rickart, {\em General theory of Banach algebras}, D. Van
Nostrand, Princeton, 1960.

\bibitem{TS} Yu.\,V. Turovskii, V.\,S. Shulman, Conditions for the massiveness of the range of a 
derivation of a Banach algebra and of associated differential operators, {\em Mat. Zametki} {\bf 42} 
(1987), 305--314; English transl., {\em Math. Notes} {\bf 42} (1987), 669--674.


\end{thebibliography}
 \end{document}